\documentclass{article}
\usepackage{amssymb,latexsym}
\usepackage{amsmath}
\usepackage{graphics}
\usepackage{graphicx}
\newtheorem{theorem}{Theorem}

\newtheorem{corollary}{Corollary}

\newtheorem{definition}{Definition}

\newtheorem{lemma}{Lemma}

\newenvironment{proof}[1][Proof]{\textbf{#1.} }{\ \rule{0.5em}{0.5em}}
\numberwithin{equation}{section} \numberwithin{theorem}{section}
\numberwithin{corollary}{section} \numberwithin{lemma}{section}
\numberwithin{remark}{section} \numberwithin{definition}{section}

\begin{document}
\title{A Note on Inextensible Flows of Curves on Oriented Surface}
\author{Onder Gokmen Yildiz$^{a}$, Soley Ersoy $^{b}$, Melek Masal $^{c} $}

\date{}

\maketitle
\begin{center}
$^{a}$ Department of Mathematics, Faculty of Arts and Sciences \\
Bilecik University, Bilecik/TURKEY

$^{b}$ Department of Mathematics, Faculty of Arts and Sciences \\
Sakarya University, Sakarya/TURKEY

$^{c}$ Department of Mathematics Teaching , Faculty of Education \\
Sakarya University,  Sakarya/TURKEY

\end{center}

\begin{abstract}
In this paper, the general formulation for
inextensible flows of curves on oriented surface in $\mathbb{R}^3
$ is investigated. The necessary and sufficient conditions for inextensible curve
flow lying an oriented surface are expressed as a partial differential
equation involving the geodesic curvature and the geodesic
torsion. Moreover, some special cases of inextensible curves on
oriented surface are given.


\end{abstract}

\section{Introduction}\label{S:intro}

It is well known that many nonlinear phenomena in physics,
chemistry and biology are described by dynamics of shapes, such as
curves and surfaces. The evolution of curve and surface has
significant applications in computer vision and image processing.
The time evolution of a curve or surface generated by its
corresponding flow in $\mathbb{R}^3 $ -for this reason we shall
also refer to curve and surface evolutions as flows throughout
this article- is said to be inextensible if, in the former case,
its arclength is preserved, and in the latter case, if its
intrinsic curvature is preserved. Physically, the inextensible curve
flows give rise to motions in which no strain energy is induced.
The swinging motion of a cord of fixed length, for example, or of
a piece of paper carried by the wind, can be described by
inextensible curve and surface flows. Such motions arise quite
naturally in a wide range of the physical applications. For example,
both Chirikjian and Burdick \cite{GC} and Mochiyama et al.
\cite{HM} study the shape control of hyper-redundant, or
snake-like robots.

\noindent The inextensible curve and surface flows also arise in the
context of many problems in computer vision \cite{MK}, \cite{HQ},
computer animation \cite{MD} and even structural mechanics
\cite{DJ}. There have been a lot of studies in the literature on
plane curve flows, particularly on evolving curves in the
direction of their curvature vector field (referred to by various
names such as "curve shortening", flow by curvature" and "heat
flow"). Particularly relevant to this paper are the methods
developed by Gage and Hamilton \cite{MGRS} and Grayson \cite{MG}
for studying the shrinking of closed plane curves to circle via
heat equation.\\
The distinction between heat flows and inextensible flows of planar curves were elaborated in detail, and some examples of the latter were given by \cite{DK1}. Also, a general formulation for inextensible flows of curves and developable surfaces in $\mathbb{R}^3 $ are exposed by \cite{DK2}.\\
In this paper, we develop the general formulation
for inextensible flows of curves  according to Darboux frame in
$\mathbb{R}^3 $. Necessary and sufficient conditions for an
inextensible curve flow are expressed as a partial differential
equation involving the geodesic curvature and geodesic torsion.

\section{Preliminaries}\label{S:intro}
Let $S$ be an oriented surface in three-dimensional Euclidean
space $E^3 $ and $\alpha \left( s \right)$ be a curve lying on the
surface $S$. Suppose that the curve $\alpha \left( s \right)$ is
spatial then there exists the Frenet frame $\left\{
{\overrightarrow T ,\overrightarrow N ,\overrightarrow B }
\right\}$ at each points of the curve where $\overrightarrow T $
is unit tangent vector, $\overrightarrow N $ is principal normal
vector and $\overrightarrow B $ is binormal vector, respectively.
The Frenet equation of the curve $\alpha \left( s \right)$ is
given by

\begin{equation*}
\begin{array}{l}
 \overrightarrow {T'}  = \kappa \overrightarrow N  \\
 \overrightarrow {N'}  =  - \kappa \overrightarrow T  + \tau \overrightarrow B  \\
 \overrightarrow {B'}  =  - \tau \overrightarrow N  \\
 \end{array}
\end{equation*}
where $\kappa $ and $\tau $ are curvature and torsion of the curve
$\alpha \left( s \right)$, respectively.

\noindent Since the curve $\alpha \left( s \right)$ lies on the
surface $S$ there exists another frame of the curve $\alpha \left(
s \right)$ which is called Darboux frame and denoted by $\left\{
{\overrightarrow T ,\overrightarrow g ,\overrightarrow n }
\right\}$. In this frame $\overrightarrow T $ is the unit tangent
of the curve, $\overrightarrow n $ is the unit normal of the
surface $S$ and $\overrightarrow g $ is a unit vector given by
$\overrightarrow g = \overrightarrow n  \times \overrightarrow T $. Since the unit tangent $\overrightarrow T $ is common element of
both Frenet frame and Darboux frame, the vectors $\overrightarrow
N ,\overrightarrow B ,\overrightarrow g $ and $\overrightarrow n $
lie on the same plane. So that the relations between these frames
can be given as follows

\begin{equation*}
\left[ \begin{array}{l}
 \overrightarrow T  \\
 \overrightarrow g  \\
 \overrightarrow n  \\
 \end{array} \right] = \left[ \begin{array}{l}
 1\,\,\,\,\,\,\,\,\,\,\,\,\,\,\,\,0\,\,\,\,\,\,\,\,\,\,\,\,\,\,\,\,\,\,0 \\
 0\,\,\,\,\,\,\,\,\,\,\,\,cos \varphi \,\,\,\,\,\,\,\,\,\sin \varphi  \\
 0\,\,\,\,\, - \sin \varphi \,\,\,\,\,\,\,\,\cos \varphi  \\
 \end{array} \right]\left[ \begin{array}{l}
 \overrightarrow T  \\
 \overrightarrow N  \\
 \overrightarrow B  \\
 \end{array} \right]
\end{equation*}
where $\varphi $ is the angle between the vectors $
\overrightarrow g $ and $\overrightarrow N $. The derivative
formulae of the Darboux frame is
\begin{equation*}
\left[ \begin{array}{l}
 \mathop {\overrightarrow T }\limits^.  \\
 \mathop {\overrightarrow g }\limits^.  \\
 \mathop {\overrightarrow n }\limits^.  \\
 \end{array} \right] = \left[ \begin{array}{l}
 \,\,\,\,\,\,\,\,0\,\,\,\,\,\,\,\,\,\,\,\,\,\,\,{k_g}\,\,\,\,\,\,\,\,\,\,\,\,{k_n} \\
  - \,\,{k_g}\,\,\,\,\,\,\,\,\,\,\,\,\,\,0\,\,\,\,\,\,\,\,\,\,\,\,\,\,{\tau _g} \\
 \,\,\,\,\,\,\,{k_n}\,\,\,\,\, - {\tau _g}\,\,\,\,\,\,\,\,\,\,\,\,\,\,0 \\
 \end{array} \right]\left[ \begin{array}{l}
 \overrightarrow T  \\
 \overrightarrow g  \\
 \overrightarrow n  \\
 \end{array} \right]
\end{equation*}
where $\ k_g ,\,k_n $ and $\tau _g $ are called the geodesic
curvature, the normal curvature and the geodesic torsions,
respectively. Here and in the following, we use "dot" to denote
the derivative with respect to the arc length parameter of a
curve.

\noindent The relations between the geodesic curvature, normal
curvature, geodesic torsion and $\kappa ,\,\tau $ are given as
follows, \cite{BO}

\begin{equation*}
\begin{array}{l}
k_g  = \kappa \cos \varphi \,,\,{\rm{   }}\,\,k_n  = \kappa \sin
\varphi \,,\,\,{\rm{  }}\,\,\tau _g  = \tau  + \frac{{d\varphi
}}{{ds}}.\,\,
 \end{array}
\end{equation*}
Furthermore, the geodesic curvature $k_g $ and geodesic torsion
$\tau _g $ of  curve $\alpha \left( s \right)$ can be calculated
as follows, \cite{BO}

\begin{equation*}
\begin{array}{l}
 k_g  = \left\langle {\frac{{d\overrightarrow \alpha  }}{{ds}},\frac{{d^2 \overrightarrow \alpha  }}{{ds^2 }} \times \overrightarrow n } \right\rangle  \\
 \tau _g  = \left\langle {\frac{{d\overrightarrow \alpha  }}{{ds}},\overrightarrow n  \times \frac{{d\overrightarrow n }}{{ds}}} \right\rangle . \\
 \end{array}
\end{equation*}
In the differential geometry of surfaces, for a curve $\alpha
\left( s \right)$ lying on a surface $S$ the following
relationships  are well-known, \cite{BO}

i-      $\alpha \left( s \right)$  is a geodesic curve if and only
if $k_g  = 0$,

ii-     $\alpha \left( s \right)$ is a asymptotic line if and only
if $k_n  = 0$,

iii-    $\alpha \left( s \right)$  is a principal line if and only
if $\tau _g  = 0$.

\noindent Through the every point of the surface a geodesic passes
in every direction. A geodesic is uniquely determined by an
initial point and tangent at that point. All straight lines on a
surface are geodesics.

\noindent Along all curved geodesics the principal normal
coincides with the surface normal. Along asymptotic lines
osculating planes and tangent planes coincide, along geodesics
they are normal. Through a point of a non-developable surface pass
two asymptotic lines which can be real or imaginary.

\section{Inextensible Flows of Curve Lying on Oriented Surface}\label{S:intro}

Throughout this paper, we suppose that
\begin{equation*}
\alpha \,\,:\left[ {0,l} \right] \times \left[ {\left. {0,w}
\right)} \right. \to M \subset E^3
\end{equation*}
is a one parameter family of differentiable curves on orientable
surface $M$ in $E^3 $, where $l$ is the arclength of the initial
curve. Let $u$ be the curve parameterization variable, $0 \le u
\le l.$ If the speed of curve $\alpha $ is denoted by $
v = \left\| {\frac{{\partial \overrightarrow \alpha  }}{{\partial u}}} \right\|
$ then the arclength of $\alpha
$ is

\begin{equation}\label{3.1}
S\left( u \right) = \int\limits_0^u {\left\| {\frac{{\partial \overrightarrow \alpha  }}{{\partial u}}} \right\|du}  = \int\limits_0^u {v\,du.}
\end{equation}
The operator $\frac{\partial }{{\partial s}}$ is given in terms of
$u$ by

\begin{equation}\label{3.2}
\begin{array}{l}
\frac{\partial }{{\partial s}} = \frac{1}{v}\frac{\partial
}{{\partial u}}.
\end{array}
\end{equation}
Thus, the arclength is $ds = v\,du.$

\begin{definition}
Let $M$ be an orientable surface and $\alpha $ be a differentiable
curve on $M$ in $E^3 $. Any flow of the curve $\alpha $ with
respect to Darboux frame $\left\{ {\overrightarrow T ,\overrightarrow g ,\overrightarrow n } \right\}$ can be
expressed following form:
\end{definition}

\begin{equation}\label{3.3}
\begin{array}{l}
\frac{{\partial \overrightarrow \alpha  }}{{\partial t}} = {f_1}\overrightarrow T  + {f_2}\overrightarrow g  + {f_3}\overrightarrow n .
 \end{array}
\end{equation}
Here, $f_1 ,\,f_2 $ and $f_3 $ are scalar speed of the curve
$\alpha .$ Let the arclength variation be

\begin{equation}\label{3.4}
\begin{array}{l}
S\left( {u,t} \right) = \int\limits_0^u {v\,du.}
 \end{array}
\end{equation}
In the Euclidean space the requirement that the curve not be
subject to any elongation or compression can be expressed by the
condition

\begin{equation}\label{3.5}
\begin{array}{l}
\frac{\partial }{{\partial t}}S\left( {u,t} \right) =
\int\limits_0^u {\frac{{\partial v}}{{\partial t}}du = 0\,\,,{\rm{
}}\,\,u \in \left[ {0,1} \right]} .
 \end{array}
\end{equation}

\begin{definition}
A curve evolution $\alpha \left( {u,t} \right)$ and its flow
$\frac{{\partial \overrightarrow \alpha  }}{{\partial t}}$
 on the oriented surface $M$
in $E^3 $ are said to be inextensible if \[\frac{\partial }{{\partial t}}\left\| {\frac{{\partial \overrightarrow \alpha  }}{{\partial u}}} \right\| = 0.\]

\end{definition}

\noindent Now, we research the necessary and sufficient condition
for inelastic curve flow. For this reason, we need to the
following Lemma.

\begin{lemma}
In ${E^3}$, let $M$ be an orientable surface and $\left\{ {\overrightarrow T ,\overrightarrow g ,\overrightarrow n } \right\}$ be a Darboux frame of $\alpha $ on $M.$ There exists
following relation between the scalar speed functions
${f_1},\,{f_2},\,{f_3}$ and the normal curvature ${k_n}$, geodesic
curvature ${k_g}$ of $\alpha $ the curve
\begin{equation}\label{3.6}
\begin{array}{l}
\frac{{\partial v}}{{\partial t}} = \frac{{\partial f_1
}}{{\partial u}} - f_2 vk_g  - f_3 vk_n .

 \end{array}
\end{equation}
\end{lemma}

\noindent \begin{proof} Since $\frac{\partial }{{\partial u}}$ and
$\frac{\partial }{{\partial t}}$ commute and ${v^2} = \left\langle {\frac{{\partial \overrightarrow \alpha  }}{{\partial u}},\frac{{\partial \overrightarrow \alpha  }}{{\partial u}}} \right\rangle ,$ we have
\begin{equation*}
\begin{array}{l}
 2v\frac{{\partial v}}{{\partial t}} = \frac{\partial }{{\partial t}}\left\langle {\frac{{\partial \overrightarrow \alpha  }}{{\partial u}},\frac{{\partial \overrightarrow \alpha  }}{{\partial u}}} \right\rangle  \\
 \,\,\,\,\,\,\,\,\,\,\,\,\,\,\, = 2\left\langle {\frac{{\partial \overrightarrow \alpha  }}{{\partial u}},\frac{\partial }{{\partial u}}\left( {{f_1}\overrightarrow T  + {f_2}\overrightarrow g  + {f_3}\overrightarrow n } \right)} \right\rangle  \\
 \,\,\,\,\,\,\,\,\,\,\,\,\,\,\, = 2v\left( {\frac{{\partial {f_1}}}{{\partial u}} - {f_2}v{k_g} - {f_3}v{k_n}} \right). \\
 \end{array}
 \end{equation*}
Thus, we reach

\begin{equation*}
\frac{{\partial v}}{{\partial t}}{{\ =}} \frac{{\partial f_1
}}{{\partial u}} - f_2 vk_g  - f_3 vk_n .
\end{equation*}
If we take in the conditions of being geodesic and asymptotic of a
curve and Lemma 3.1, we give the following
\end{proof}

\begin{corollary}
If the curve is a geodesic curve or asymptotic curve, then there
is following equations
\end{corollary}
\begin{equation*}
\begin{array}{l}
\frac{{\partial v}}{{\partial t}} = \frac{{\partial
{f_1}}}{{\partial u}} - {f_3}v{k_n}
 \end{array}
\end{equation*}
or
\begin{equation*}
\begin{array}{l}
\frac{{\partial v}}{{\partial t}} = \frac{{\partial
{f_1}}}{{\partial u}} - {f_2}v{k_g},
\end{array}
\end{equation*}
respectively.
\begin{theorem}\label{T:3.1}
Let $\left\{ {\overrightarrow T ,\overrightarrow g ,\overrightarrow n } \right\}$ be Darboux frame of the curve
$\alpha $ on $M$ and $\frac{{\partial \overrightarrow \alpha  }}{{\partial t}} = {f_1}\overrightarrow T  + {f_2}\overrightarrow g  + {f_3}\overrightarrow n $ be a differentiable flow of $\alpha $ in
${\mathbb{R}^3}$ . Then the flow is inextensible if and only if
\begin{equation}\label{3.7}
\begin{array}{l}
\frac{{\partial {f_1}}}{{\partial s}} = {f_2}{k_g} + {f_3}{k_n}.
\end{array}
\end{equation}
\end{theorem}

\noindent \begin{proof} Suppose that the curve flow is
inextensible. From equations (\ref{3.4}) and (\ref{3.6}) for $u \in \left[
{0,l} \right]$ , we see that
\begin{equation}\label{3.8}
\frac{\partial }{{\partial t}}S\left( {u,t} \right) =
\int\limits_0^u {\frac{{\partial v}}{{\partial t}}du = }
\int\limits_0^u {\left( {\frac{{\partial {f_1}}}{{\partial u}} -
{f_2}v{k_g} - {f_3}v{k_n}} \right)\,du}  = 0.
\end{equation}
Thus, it can be seen that
\begin{equation}\label{3.9}
\frac{{\partial {f_1}}}{{\partial u}} = {f_2}v{k_g} + {f_3}v{k_n}.
\end{equation}
Considering the last equation and (\ref{3.2}), we reach
\begin{equation*}
\begin{array}{l}
\frac{{\partial f_1 }}{{\partial s}} = f_2 k_g  + f_3 k_n .
 \end{array}
\end{equation*}

\noindent Conversely, following similar way as above, the proof is
completed.
\end{proof}

\noindent From Theorem 3.1, we have following corollary.

\begin{corollary}

i-  Let the curve $\alpha $ is a geodesic curve on $M.$ Then the
curve flow is inextensible if and only if  $\frac{{\partial
{f_1}}}{{\partial s}} = {f_3}{k_n}$.

ii- Let the curve $\alpha $ is a asymptotic line on $M.$ Then the
curve flow is inextensible if and only if  $\frac{{\partial
{f_1}}}{{\partial s}} = {f_2}{k_g}.$
\end{corollary}

\noindent Now, we restrict ourselves to the arclength parameterized
curves. That is, $v = 1$ and the local coordinate $u$ corresponds
to the curve arclength $s$. We require the following Lemma.

\begin{lemma}
Let $M$ be an orientable surface in ${E^3}$ and $\left\{ {\overrightarrow T ,\overrightarrow g ,\overrightarrow n } \right\}$ be a Darboux frame of the curve $\alpha $ on $M.$ Then,
the differentiations of $\left\{ {\overrightarrow T ,\overrightarrow g ,\overrightarrow n } \right\}$ with respect to
$t$ is
\begin{equation*}
\begin{array}{l}
 \frac{{\partial \overrightarrow T }}{{\partial t}} = \left( {{f_1}{k_g} + \frac{{\partial {f_2}}}{{\partial s}} - {f_3}{\tau _g}} \right)\overrightarrow g  + \left( {{f_1}{k_n} + \frac{{\partial {f_3}}}{{\partial s}} + {f_2}{\tau _g}} \right)\overrightarrow n  \\
 \frac{{\partial \overrightarrow g }}{{\partial t}} =  - \left( {{f_1}{k_g} + \frac{{\partial {f_2}}}{{\partial s}} - {f_3}{\tau _g}} \right)\overrightarrow T  + \psi \overrightarrow n  \\
 \frac{{\partial \overrightarrow n }}{{\partial t}} =  - \left( {{f_1}{k_n} + \frac{{\partial {f_3}}}{{\partial s}} + {f_2}{\tau _g}} \right)\overrightarrow T  - \psi \overrightarrow g  \\
 \end{array}
\end{equation*}
where  $\psi  = \left\langle {\frac{{\partial \overrightarrow g }}{{\partial t}},\overrightarrow n } \right\rangle .$
\end{lemma}
\begin{proof}
Since $\frac{\partial }{{\partial t}}$ and $\frac{\partial
}{{\partial s}}$ are commutative, it seen that
\begin{equation*}
\begin{array}{l}
 \frac{{\partial \overrightarrow T }}{{\partial t}} = \frac{\partial }{{\partial t}}\left( {\frac{{\partial \overrightarrow \alpha  }}{{\partial s}}} \right) = \frac{\partial }{{\partial s}}\left( {\frac{{\partial \overrightarrow \alpha  }}{{\partial t}}} \right) = \frac{\partial }{{\partial s}}\left( {{f_1}\overrightarrow T  + {f_2}\overrightarrow g  + {f_3}\overrightarrow n } \right) \\
 \,\,\,\,\,\,\,\,\,\,\, = \,\frac{{\partial {f_1}}}{{\partial s}}\overrightarrow T  + {f_1}\frac{{\partial \overrightarrow T }}{{\partial s}} + \frac{{\partial {f_2}}}{{\partial s}}\overrightarrow g  + {f_2}\frac{{\partial \overrightarrow g }}{{\partial s}} + \frac{{\partial {f_3}}}{{\partial s}}\overrightarrow n  + {f_3}\frac{{\partial \overrightarrow n }}{{\partial s}}. \\
 \end{array}
\end{equation*}
Substituting the equation (\ref{3.7}) into the last equation and using
Theorem \ref{T:3.1}, we have
\begin{equation*}
\frac{{\partial \overrightarrow T }}{{\partial t}} = \left( {{f_1}{k_g} + \frac{{\partial {f_2}}}{{\partial s}} - {f_3}{\tau _g}} \right)\overrightarrow g  + \left( {{f_1}{k_n} + \frac{{\partial {f_3}}}{{\partial s}} + {f_2}{\tau _g}} \right)\overrightarrow n .
\end{equation*}
\noindent Now, let us differentiate the Darboux frame with respect
to $t$ as follows;
\begin{equation*}
\,\,\,\,\,\,\,0 = \frac{\partial }{{\partial t}}\left\langle {\overrightarrow T ,\overrightarrow g } \right\rangle  = \left\langle {\frac{{\partial \overrightarrow T }}{{\partial t}},\overrightarrow g } \right\rangle  + \left\langle {\overrightarrow T ,\frac{{\partial \overrightarrow g }}{{\partial t}}} \right\rangle
\end{equation*}
\begin{equation}\label{3.10}
\,\, = \left( {{f_1}{k_g} + \frac{{\partial {f_2}}}{{\partial s}} - {f_3}{\tau _g}} \right) + \left\langle {\overrightarrow T ,\frac{{\partial \overrightarrow g }}{{\partial t}}} \right\rangle
\end{equation}
\begin{equation*}
\,\,\,\,\,\,\,0\, = \frac{\partial }{{\partial t}}\left\langle {\overrightarrow T ,\overrightarrow n } \right\rangle  = \left\langle {\frac{{\partial \overrightarrow T }}{{\partial t}},\overrightarrow n } \right\rangle  + \left\langle {\overrightarrow T ,\frac{{\partial \overrightarrow n }}{{\partial t}}} \right\rangle
\end{equation*}
\begin{equation}\label{3.11}
\,\, = \left( {{f_1}{k_n} + \frac{{\partial {f_3}}}{{\partial s}} + {f_2}{\tau _g}} \right) + \left\langle {\overrightarrow T ,\frac{{\partial \overrightarrow n }}{{\partial t}}} \right\rangle
\end{equation}
From (\ref{3.10}) and (\ref{3.11}), we have obtain
\begin{equation*}
\frac{{\partial \overrightarrow g }}{{\partial t}}\,\, =  - \left( {{f_1}{k_g} + \frac{{\partial {f_2}}}{{\partial s}} - {f_3}{\tau _g}} \right)\overrightarrow T  + \psi \overrightarrow n
\end{equation*}
and
\begin{equation*}
\frac{{\partial \overrightarrow n }}{{\partial t}}\,\, =  - \left( {{f_1}{k_n} + \frac{{\partial {f_3}}}{{\partial s}} + {f_2}{\tau _g}} \right)\overrightarrow T  - \psi \overrightarrow g
\end{equation*}
respectively, where $\psi  = \left\langle {\frac{{\partial \overrightarrow g }}{{\partial t}},\overrightarrow n } \right\rangle .$
\end{proof}

\noindent If we take into consideration last Lemma, we have
following corollary.
\begin{corollary}
Let $M$ be an orientable surface in ${E^3}$.

i-  If the curve $\alpha $ is a geodesic curve, then
\begin{equation*}
\begin{array}{l}
 \frac{{\partial \overrightarrow T }}{{\partial t}} = \left( {\frac{{\partial {f_2}}}{{\partial s}} - {f_3}{\tau _g}} \right)\overrightarrow g  + \left( {{f_1}{k_n} + \frac{{\partial {f_3}}}{{\partial s}} + {f_2}{\tau _g}} \right)\overrightarrow n , \\
 \frac{{\partial \overrightarrow g }}{{\partial t}} =  - \left( {\frac{{\partial {f_2}}}{{\partial s}} - {f_3}{\tau _g}} \right)\overrightarrow T  + \psi \overrightarrow n , \\
 \frac{{\partial \overrightarrow n }}{{\partial t}} =  - \left( {{f_1}{k_n} + \frac{{\partial {f_3}}}{{\partial s}} + {f_2}{\tau _g}} \right)\overrightarrow T  - \psi \overrightarrow g , \\
 \end{array}
\end{equation*}
where $\psi  = \left\langle {\frac{{\partial \overrightarrow g }}{{\partial t}},\overrightarrow n } \right\rangle .$

ii- If the curve $\alpha $ is a asymptotic line, then
\begin{equation*}
\begin{array}{l}
 \frac{{\partial \overrightarrow T }}{{\partial t}} = \left( {{f_1}{k_g} + \frac{{\partial {f_2}}}{{\partial s}} - {f_3}{\tau _g}} \right)\overrightarrow g  + \left( {\frac{{\partial {f_3}}}{{\partial s}} + {f_2}{\tau _g}} \right)\overrightarrow n , \\
 \frac{{\partial \overrightarrow g }}{{\partial t}} =  - \left( {{f_1}{k_g} + \frac{{\partial {f_2}}}{{\partial s}} - {f_3}{\tau _g}} \right)\overrightarrow T  + \psi \overrightarrow n , \\
 \frac{{\partial \overrightarrow n }}{{\partial t}} =  - \left( {\frac{{\partial {f_3}}}{{\partial s}} + {f_2}{\tau _g}} \right)\overrightarrow T  - \psi \overrightarrow g , \\
 \end{array}
\end{equation*}
where $\psi  = \left\langle {\frac{{\partial \overrightarrow g }}{{\partial t}},\overrightarrow n } \right\rangle .$

iii-    If the curve   is a curvature line, then
\begin{equation*}
\begin{array}{l}
 \frac{{\partial \overrightarrow T }}{{\partial t}} = \left( {{f_1}{k_g} + \frac{{\partial {f_2}}}{{\partial s}}} \right)\overrightarrow g  + \left( {{f_1}{k_n} + \frac{{\partial {f_3}}}{{\partial s}}} \right)\overrightarrow n , \\
 \frac{{\partial \overrightarrow g }}{{\partial t}} =  - \left( {{f_1}{k_g} + \frac{{\partial {f_2}}}{{\partial s}}} \right)\overrightarrow T  + \psi \overrightarrow n , \\
 \frac{{\partial \overrightarrow n }}{{\partial t}} =  - \left( {{f_1}{k_n} + \frac{{\partial {f_3}}}{{\partial s}}} \right)\overrightarrow T  - \psi \overrightarrow g , \\
 \end{array}
\end{equation*}
where $\psi  = \left\langle {\frac{{\partial \overrightarrow g }}{{\partial t}},\overrightarrow n } \right\rangle .$
\end{corollary}

\begin{theorem}
Suppose that the curve flow $\frac{{\partial \overrightarrow \alpha  }}{{\partial t}} = {f_1}\overrightarrow T  + {f_2}\overrightarrow g  + {f_3}\overrightarrow n $ is inextensible on the orientable
surface on $M$. In this case, the following partial differential
equation are held:

$\frac{{\partial {k_g}}}{{\partial t}} = \frac{{\partial
{f_1}}}{{\partial s}}{k_g} + {f_1}\frac{{\partial
{k_g}}}{{\partial s}} + \frac{{{\partial ^2}{f_2}}}{{\partial
{s^2}}} - \frac{{\partial {f_3}}}{{\partial s}}{\tau _g} -
{f_3}\frac{{\partial {\tau _g}}}{{\partial s}} - {f_1}{k_n}{\tau
_g} - \frac{{\partial {f_3}}}{{\partial s}}{\tau _g} - {f_2}\tau
_g^2$ ,

$\frac{{\partial {k_n}}}{{\partial t}} = \frac{{\partial
{f_1}}}{{\partial s}}{k_n} + {f_1}\frac{{\partial
{k_n}}}{{\partial s}} + \frac{{{\partial ^2}{f_3}}}{{\partial
{s^2}}} + \frac{{\partial {f_2}}}{{\partial s}}{\tau _g} +
{f_2}\frac{{\partial {\tau _g}}}{{\partial s}} + {f_1}{k_g}{\tau
_g} + \frac{{\partial {f_2}}}{{\partial s}}{\tau _g} - {f_3}\tau
_g^2$ ,

$\frac{{\partial {\tau _g}}}{{\partial t}} =  - {f_1}{k_g}{k_n} -
\frac{{\partial {f_2}}}{{\partial s}}{k_n} + {f_3}{k_n}{\tau _g} +
\frac{{\partial \psi }}{{\partial s}}$,

$\psi {k_n} = \left( { - {f_1}{k_n} - \frac{{\partial
{f_3}}}{{\partial s}} - {f_2}{\tau _g}} \right){\tau _g}$,

$\psi {k_g} = \left( { - {f_1}{k_g} - \frac{{\partial
{f_3}}}{{\partial s}} + {f_3}{\tau _g}} \right){\tau _g}$ .

\end{theorem}\label{T:3.2}
\begin{proof}
Since $\frac{\partial }{{\partial s}}\frac{{\partial \overrightarrow T }}{{\partial t}} = \frac{\partial }{{\partial t}}\frac{{\partial \overrightarrow T }}{{\partial s}}$ we get
\begin{equation*}
\begin{array}{l}
 \frac{\partial }{{\partial s}}\frac{{\partial \overrightarrow T }}{{\partial t}} = \frac{\partial }{{\partial s}}\left[ {\left( {{f_1}{k_g} + \frac{{\partial {f_2}}}{{\partial s}} - {f_3}{\tau _g}} \right)\overrightarrow g  + \left( {{f_1}{k_n} + \frac{{\partial {f_3}}}{{\partial s}} + {f_2}{\tau _g}} \right)\overrightarrow n } \right] \\
 \,\,\,\,\,\,\,\,\,\,\,\,\,\,\,\,\, = \,\left( {\frac{{\partial {f_1}}}{{\partial s}}{k_g} + {f_1}\frac{{\partial {k_g}}}{{\partial s}} + \frac{{{\partial ^2}{f_2}}}{{\partial {s^2}}} - \frac{{\partial {f_3}}}{{\partial s}}{\tau _g} - {f_3}\frac{{\partial {\tau _g}}}{{\partial s}}} \right)\overrightarrow g  + \left( {{f_1}{k_g} + \frac{{\partial {f_2}}}{{\partial s}} - {f_3}{\tau _g}} \right)\frac{{\partial \overrightarrow g }}{{\partial s}} \\
 \,\,\,\,\,\,\,\,\,\,\,\,\,\,\,\,\,\,\,\,\,\,\, + \left( {\frac{{\partial {f_1}}}{{\partial s}}{k_n} + {f_1}\frac{{\partial {k_n}}}{{\partial s}} + \frac{{{\partial ^2}{f_3}}}{{\partial {s^2}}} + \frac{{\partial {f_2}}}{{\partial s}}{\tau _g} + {f_2}\frac{{\partial {\tau _g}}}{{\partial s}}} \right)\overrightarrow n  + \left( {{f_1}{k_n} + \frac{{\partial {f_3}}}{{\partial s}} + {f_2}{\tau _g}} \right)\frac{{\partial \overrightarrow n }}{{\partial s}} \\
 \end{array}
\end{equation*}
i.e.,
\begin{equation*}
\begin{array}{l}
 \frac{\partial }{{\partial s}}\frac{{\partial \overrightarrow T }}{{\partial t}} = \,\left( {\frac{{\partial {f_1}}}{{\partial s}}{k_g} + {f_1}\frac{{\partial {k_g}}}{{\partial s}} + \frac{{{\partial ^2}{f_2}}}{{\partial {s^2}}} - \frac{{\partial {f_3}}}{{\partial s}}{\tau _g} - {f_3}\frac{{\partial {\tau _g}}}{{\partial s}}} \right)\overrightarrow g  + \left( {{f_1}{k_g} + \frac{{\partial {f_2}}}{{\partial s}} - {f_3}{\tau _g}} \right)\left( { - {k_g}\overrightarrow T  + {\tau _g}\overrightarrow n } \right) \\
 \,\,\,\,\,\,\,\,\,\,\,\,\,\,\,\,\,\,\,\,\,\, + \left( {\frac{{\partial {f_1}}}{{\partial s}}{k_n} + {f_1}\frac{{\partial {k_n}}}{{\partial s}} + \frac{{{\partial ^2}{f_3}}}{{\partial {s^2}}} + \frac{{\partial {f_2}}}{{\partial s}}{\tau _g} + {f_2}\frac{{\partial {\tau _g}}}{{\partial s}}} \right)\overrightarrow n  + \left( {{f_1}{k_n} + \frac{{\partial {f_3}}}{{\partial s}} + {f_2}{\tau _g}} \right)\left( { - {k_g}\overrightarrow T  - {\tau _g}\overrightarrow g } \right) \\
 \end{array}
\end{equation*}
while
\begin{equation*}
\frac{\partial }{{\partial t}}\frac{{\partial \overrightarrow T }}{{\partial s}} = \frac{\partial }{{\partial t}}\left( {{k_g}\overrightarrow g  + {k_n}\overrightarrow n } \right) = \frac{{\partial {k_g}}}{{\partial t}}\overrightarrow g  + {k_g}\frac{{\partial \overrightarrow g }}{{\partial t}} + \frac{{\partial {k_n}}}{{\partial t}}\overrightarrow n  + {k_n}\frac{{\partial \overrightarrow n }}{{\partial t}}.
\end{equation*}
Thus, from the both of above two equations, we reach
\begin{equation}\label{3.12}
\frac{{\partial {k_g}}}{{\partial t}} = \frac{{\partial
{f_1}}}{{\partial s}}{k_g} + {f_1}\frac{{\partial
{k_g}}}{{\partial s}} + \frac{{{\partial ^2}{f_2}}}{{\partial
{s^2}}} - \frac{{\partial {f_3}}}{{\partial s}}{\tau _g} -
{f_3}\frac{{\partial {\tau _g}}}{{\partial s}} - {f_1}{k_n}{\tau
_g} - \frac{{\partial {f_3}}}{{\partial s}}{\tau _g} - {f_2}{\tau
_g}^2
\end{equation}
and
\begin{equation}\label{3.13}
\frac{{\partial {k_n}}}{{\partial t}} = \frac{{\partial
{f_1}}}{{\partial s}}{k_n} + {f_1}\frac{{\partial
{k_n}}}{{\partial s}} + \frac{{{\partial ^2}{f_3}}}{{\partial
{s^2}}} + \frac{{\partial {f_2}}}{{\partial s}}{\tau _g} +
{f_2}\frac{{\partial {\tau _g}}}{{\partial s}} + {f_1}{k_g}{\tau
_g} + \frac{{\partial {f_2}}}{{\partial s}}{\tau _g} - {f_3}{\tau
_g}^2.
\end{equation}
Noting that $\frac{\partial }{{\partial s}}\frac{{\partial \overrightarrow g }}{{\partial t}} = \frac{\partial }{{\partial t}}\frac{{\partial \overrightarrow g }}{{\partial s}}$, it is seen that

$\begin{array}{l}
 \frac{\partial }{{\partial s}}\frac{{\partial \overrightarrow g }}{{\partial t}} = \frac{\partial }{{\partial s}}\left[ { - \left( {{f_1}{k_g} + \frac{{\partial {f_2}}}{{\partial s}} - {f_3}{\tau _g}} \right)\overrightarrow T  + \psi \overrightarrow n } \right] \\
 \,\,\,\,\,\,\,\,\,\,\,\,\,\,\,\,\, = \, - \left( {\frac{{\partial {f_1}}}{{\partial s}}{k_g} + {f_1}\frac{{\partial {k_g}}}{{\partial s}} + \frac{{{\partial ^2}{f_2}}}{{\partial {s^2}}} - \frac{{\partial {f_3}}}{{\partial s}}{\tau _g} - {f_3}\frac{{\partial {\tau _g}}}{{\partial s}}} \right)\overrightarrow T  - \left( {{f_1}{k_g} + \frac{{\partial {f_2}}}{{\partial s}} - {f_3}{\tau _g}} \right)\left( {{k_g}\overrightarrow g  + {k_n}\overrightarrow n } \right) \\
 \,\,\,\,\,\,\,\,\,\,\,\,\,\,\,\,\,\,\,\,\,\,\, + \frac{{\partial \psi }}{{\partial s}}n + \psi \left( { - {k_n}\overrightarrow T  - {\tau _g}\overrightarrow g } \right) \\
 \end{array}$
while
\begin{equation*}
\frac{\partial }{{\partial t}}\frac{{\partial \overrightarrow g }}{{\partial s}} = \frac{\partial }{{\partial t}}\left( { - {k_g}\overrightarrow T  + {\tau _g}\overrightarrow n } \right) =  - \frac{{\partial {k_g}}}{{\partial t}}\overrightarrow T  - {k_g}\frac{{\partial \overrightarrow T }}{{\partial t}} + \frac{{\partial {\tau _g}}}{{\partial t}}\overrightarrow n  + {\tau _g}\frac{{\partial \overrightarrow n }}{{\partial t}}.
\end{equation*}
Thus, we obtain
 \begin{equation}\label{3.14}
\frac{{\partial {k_g}}}{{\partial t}} = \frac{{\partial
{f_1}}}{{\partial s}}{k_g} + {f_1}\frac{{\partial
{k_g}}}{{\partial s}} + \frac{{{\partial ^2}{f_2}}}{{\partial
{s^2}}} - \frac{{\partial {f_3}}}{{\partial s}}{\tau _g} -
{f_3}\frac{{\partial {\tau _g}}}{{\partial s}} + \psi {k_n}
\end{equation}
and
\begin{equation}\label{3.15}
\frac{{\partial {\tau _g}}}{{\partial t}} =  - {f_1}{k_g}{k_n} -
\frac{{\partial {f_2}}}{{\partial s}}{k_n} + {f_3}{k_n}{\tau _g} +
\frac{{\partial \psi }}{{\partial s}}.
\end{equation}
From the equations (\ref{3.12}) and (\ref{3.14}), it is seen that
\begin{equation*}
\begin{array}{l}
 \psi {k_n} =  - {f_1}{k_n}{\tau _g} - \frac{{\partial {f_3}}}{{\partial s}}{\tau _g} - {f_2}{\tau _g}^2 \\
 \,\,\,\,\,\,\,\,\,\,\,\, = \,\left( { - {f_1}{k_n} - \frac{{\partial {f_3}}}{{\partial s}} - {f_2}{\tau _g}} \right){\tau _g}. \\
 \end{array}
\end{equation*}
By same way as above and considering $\frac{\partial }{{\partial s}}\frac{{\partial \overrightarrow n }}{{\partial t}} = \frac{\partial }{{\partial t}}\frac{{\partial \overrightarrow n }}{{\partial s}}$, we reach
\begin{equation}\label{3.16}
\frac{{\partial {k_n}}}{{\partial t}} = \frac{{\partial
{f_1}}}{{\partial s}}{k_n} + {f_1}\frac{{\partial
{k_n}}}{{\partial s}} + \frac{{{\partial ^2}{f_3}}}{{\partial
{s^2}}} + \frac{{\partial {f_2}}}{{\partial s}}{\tau _g} +
{f_2}\frac{{\partial {\tau _g}}}{{\partial s}} - \psi {k_g}.
\end{equation}
Hence, from the equations (\ref{3.13}) and (\ref{3.16}), we get
\begin{equation*}
\psi {k_g} = \,\left( { - {f_1}{k_g} - \frac{{\partial
{f_2}}}{{\partial s}} + {f_3}{\tau _g}} \right){\tau _g}.
\end{equation*}
\end{proof}

\noindent Thus, we give the following corollary from last theorem.

\begin{corollary}
Let $M$ be an orientable surface in $E^3 .$

i-  If the curve $\alpha $ is a geodesic curve on $M,$ then we
have
\begin{equation*}
\frac{{\partial {k_n}}}{{\partial t}} = \frac{{\partial
{f_1}}}{{\partial s}}{k_n} + {f_1}\frac{{\partial
{k_n}}}{{\partial s}} + \frac{{{\partial ^2}{f_3}}}{{\partial
{s^2}}} + \frac{{\partial {f_2}}}{{\partial s}}{\tau _g} +
{f_2}\frac{{\partial {\tau _g}}}{{\partial s}} + \frac{{\partial
{f_2}}}{{\partial s}}{\tau _g} - {f_3}{\tau _g}^2
\end{equation*}
\begin{equation*}
\frac{{\partial {\tau _g}}}{{\partial t}} =  - \frac{{\partial
{f_2}}}{{\partial s}}{k_n} + {f_3}{k_n}{\tau _g} + \frac{{\partial
\psi }}{{\partial s}}
\end{equation*}
and
\begin{equation*}
\psi {k_n} = \left( { - {f_1}{k_n} - \frac{{\partial
{f_3}}}{{\partial s}} - {f_2}{\tau _g}} \right){\tau _g}.
\end{equation*}
ii- If the curve $\alpha $ is a asymptotic line, we have

\begin{equation*}
\frac{{\partial {k_g}}}{{\partial t}} = \frac{{\partial
{f_1}}}{{\partial s}}{k_g} + {f_1}\frac{{\partial
{k_g}}}{{\partial s}} + \frac{{{\partial ^2}{f_2}}}{{\partial
{s^2}}} - \frac{{\partial {f_3}}}{{\partial s}}{\tau _g} -
{f_3}\frac{{\partial {\tau _g}}}{{\partial s}} - \frac{{\partial
{f_3}}}{{\partial s}}{\tau _g} - {f_2}{\tau _g}^2
\end{equation*}
\begin{equation*}
\frac{{\partial {\tau _g}}}{{\partial t}} = \frac{{\partial \psi
}}{{\partial s}}
\end{equation*}
and
\begin{equation*}
\psi {k_g} = \left( { - {f_1}{k_g} - \frac{{\partial
{f_2}}}{{\partial s}} + {f_3}{\tau _g}} \right){\tau _g}.
\end{equation*}
iii-    If the curve $\alpha $ is a curvature line, then we have
\begin{equation*}
\begin{array}{l}
 \frac{{\partial {k_g}}}{{\partial t}} = \frac{{\partial {f_1}}}{{\partial s}}{k_g} + {f_1}\frac{{\partial {k_g}}}{{\partial s}} + \frac{{{\partial ^2}{f_2}}}{{\partial {s^2}}} \\
 \frac{{\partial {k_n}}}{{\partial t}} = \frac{{\partial {f_1}}}{{\partial s}}{k_n} + {f_1}\frac{{\partial {k_n}}}{{\partial s}} + \frac{{{\partial ^2}{f_3}}}{{\partial {s^2}}}. \\
 \end{array}
\end{equation*}
\end{corollary}

\end {document}